\documentclass{amsart}
\usepackage{amsmath, amsfonts, amssymb}
\usepackage{latexsym}
\usepackage{color}
\usepackage{amsfonts}
\usepackage{graphicx}
\usepackage{amsthm, amsmath}   
\usepackage{xypic}
\usepackage{graphicx}

\newtheorem{thm}{Theorem}[section] 
\newtheorem{lemma}[thm]{Lemma}
\newtheorem{prop}[thm]{Proposition}
\newtheorem{cor}[thm]{Corollary}

\theoremstyle{definition}
\newtheorem{remark}[thm]{Remark}

\newtheorem{defn}[thm]{Definition}

\newcommand{\cO}{{\mathcal O}}

\newcommand{\ignore}[1]{}

\def\F{{\mathbb F}}

\def\qed{\hfill\vbox{\hrule\hbox{\vrule\kern3pt\vbox{\kern6pt}\kern3pt\vrule}\hrule}\bigskip}

\def\qed{\hfill\vbox{\hrule\hbox{\vrule\kern3pt\vbox{\kern6pt}\kern3pt\vrule}\hrule}\bigskip}

\theoremstyle{plain}

\theoremstyle{definition}

\begin{document}

\title{Formal properties in small codimension}


\author{Jorge Caravantes}
\address{Departamento de \'Algebra, Facultad de Matem\'aticas, Plaza de Ciencias, Universidad Complutense de Madrid. 28040 Madrid, Spain}
\email{jcaravan@mat.ucm.es}
\thanks{The first author was partially supported by the Spanish Ministerio de Economíay Competitividad and by the European Regional Development Fund (ERDF), under the project MTM2011-25816-C02-(02).}

\author{Nicolas Perrin}
\address{Mathematisches Institut, Heinrich-Heine-Universit{\"a}t,
D-40204 D{\"u}sseldorf, Germany}
\email{perrin@math.uni-duesseldorf.de}



\begin{abstract}
In this note we extend connectedness results to formal properties
of inverse images under proper maps of Schubert varieties and of the diagonal in products of projective rational homogeneous spaces.
\end{abstract}

\maketitle

\markboth{J. CARAVANTES \& N.~PERRIN}{Formal properties in small
  codimension}

\section*{Introduction}\label{sec:intro}

That small codimension subvarieties $Y$ in a \emph{sufficently
  positive} variety $X$ inherit some of the properties of $X$ is a
well know phenomenon. In this note we study the formal properties of
small codimension 
subvarieties in projective rational homogeneous spaces. Indeed, many
of the geometric properties of a subvariety $Y$ in $X$ can be encoded
in the \emph{algebraic tubular neighbourhood} $X_Y$ of $Y$ in $X$. In
particular if $K(X_Y)$ is the ring of formal functions of $X$
along $Y$ (see Definition \ref{KXY}), then, for $X$ normal, the
subvariety $Y$ is connected if and only if $K(X_Y)$ is a field. If
furthermore $K(X_Y)$ is isomorphic to the field $K(X)$ of rational
functions, then $Y$ is called G3 in $X$. In this note we prove the G3
property for some small codimension subvarieties in projective
rational homogeneous spaces.

More precisely, let $G$ be a reductive group, let $P$ be a parabolic
subgroup and let $X = G/P$. Denote by $X^P(w)$ the Schubert
varieties in $X$ (see below for more
details). We define (see Definition \ref{admis}) admissible Schubert
subvarieties $X^P(v)$ in $X^P(w)$ and prove 

\begin{thm}\label{main0}
  Let $f:Y \to X$ be a proper morphism and let $X^P(v)$ be admissible
  in $X^P(w)$ with $[f(Y)] \cdot [X^P(v)] \ne
  0$. Then $f^{-1}(X^P(w))$ is G3 in $Y$.
\end{thm}

Applying this result we obtain formal properties of the inverse image
of the diagonal in partial flag varieties: let $X = \F((a_i)_{i \in [1,r]}
; V)$ be a variety of partial flags in $V$ (see Section
\ref{section-flag}) and let $H$ be a hyperplane of $V$. Write
$\Delta_H$ for the diagonal in $\F((a_i)_{i \in [1,r]} ; H) \times
\F((a_i)_{i \in [1,r]} ; H)$ and $\Delta$ for the diagonal in $X
\times X$. 

\begin{cor}
Let $f:Y\to X\times X$ be a proper morphism with $[f(Y)] \cdot
[\Delta_H] \ne 0$. Then $f^{-1}(\Delta)$ is $G3$ in $Y$. 
\end{cor}

\subsection*{Notation}
\label{notation}


We work over an algebraically closed field $k$ of characteristic
0. Let $G$ be a reductive 
group over $k$ and let $T$ be a maximal torus of $G$. Let $B$ a Borel 
subgroup containing $T$ and $P$ a parabolic subgroup containing
$B$. Denote by $W$ and $W_P$ the Weyl groups of $G$ and $P$ and by
$W^P$ the set of minimal length representatives of $W/W_P$. The
$B$-orbit closures in $X$ are called Schubert varieties and are
indexed by elements $w \in W^P$. We write $X^P(w)$ for the
corresponding Schubert variety. Given $\alpha$ in the root system of
$(G,T)$, denote by $U(\alpha)$ the corresponding unipotent subgroup
and by $s_\alpha$ the corresponding reflection in $W$. 

We write $\Sigma(P)$ for the set of simple roots $\alpha$ such that
$U(-\alpha)$ is not contained in $P$. Define $S^P(w)$ as the
stabiliser in $G$ of $X^P(w)$. Then $S^P(w)$ is a parabolic subgroup
of $G$ containing $B$. Set $\Sigma^P(w) = \Sigma(S^P(w))$. For more details on rational homogeneous spaces and algebraic groups, we refer to \cite{Br} and  \cite{S} respectively.

\section{Preliminaries regarding the G3 property}

\subsection{Ring of formal functions}

Let $X$ be a scheme and $Y$ be a closed subscheme defined by the sheaf
of ideals $I$. Write $Y_n = (Y,\cO_X/I^n)$ for the infinitesimal
neighbourhood of order $n$ of $Y$ in $X$. The formal completion of $X$
along $Y$ is the formal scheme $\displaystyle{X_{Y}=\lim_\to Y_n}$. 

\begin{defn}
\label{KXY}
The ring of formal rational functions $K(X_Y)$ of $X$ along $Y$ is the
sheaf associated to the presheaf $U \mapsto [\cO_{X_Y}(U)]_0$ where
$[A]_0$ denotes the total ring of fractions of $A$.  
\end{defn}

The ring $K(X_Y)$ is in general not a field. However we have the
following result. 

\begin{prop}[{see \cite[Corollary 9.10]{B1}}]
\label{connectfield}
Let $X$ be an algebraic variety and let $Y$ be a closed subvariety
$X$. Let $u:X'\to X$ be the normalization of $X$. Then $K(X_{Y})$ is a
field if and only if $u^{-1}(Y)$ is connected.
\end{prop}

There is a natural morphism $\alpha_{X,Y} : K(X) \to K(X_Y)$ (see
\cite[Page 84]{B1}). 

\begin{defn}
The subscheme $Y$ is called $G3$ in $X$ if $\alpha_{X,Y}$ is an isomorphism.
\end{defn}

We shall also use the following two results for proving the G3 property.

\begin{prop}\label{HiroMatsu}(see \cite{HM} or \cite[Corollary 9.13]{B1})
Let $f:X'\to X$ be a proper surjective morphism. Suppose $Y$ is G3 in
$X$. Then $f^{-1}(Y)$ is G3 in $X'$. 
\end{prop}

\begin{prop}\label{Badescu}(see \cite[Proposition 9.23]{B1})
Let $f:X'\to X$ be a proper surjective morphism of irreducible
varieties and let $Y\subset X$
and $Y'\subset X'$ be closed subvarieties such that  $f(Y')\subset
Y$. Assume that the rings $K(X_{Y})$ and $K(X'_{f^{-1}(Y)})$ are
both fields. If $Y'$ is G3 in $X'$, then $Y$ is G3 in $X$. 
\end{prop}

\subsection{Generating subvarieties}

Let $X = G/P$ 
with $G$ a reductive group and $P$ a parabolic subgroup. 

\begin{defn}
Let $Y$ be a closed irreducible subvariety in $X$.

1. Let $y \in Y$ and $\varphi : G \to Y,\ g \mapsto g\cdot y$. We
define $G_Y$ to be the subgroup generated by $ \varphi^{-1}(G)$. One
easily checks that this definition does not depend on the choice of
$y$. 

2. The subvariety $Y$ generates $X$ if $G_Y = G$
\end{defn}

\begin{prop}[{see \cite[Corollary 13.8]{B1}}]
A generating subvariety of $X$ is G3.
\end{prop}

\begin{defn}
A Schubert variety $X(t)$ with $t\in W^P$ is called minimal generating
if $t$ has a reduced expression $s_{\beta_1} \cdots s_{\beta_k}$ such
that the roots $(\beta_i)_{i \in [1,k]}$ are all simple distinct and
$\Sigma(P) = \{\beta_1 , \cdots , \beta_k\}$.  
\end{defn}

\begin{lemma}
\label{lemma-gen}
A minimal generating Schubert variety is a generating subvariety.
\end{lemma}

\begin{proof}
Let $Y = X^P(t)$ be a minimal generating Schubert variety with reduced
expression $t = s_{\beta_1} \cdots s_{\beta_k}$ such that the roots
$(\beta_i)_{i \in [1,k]}$ are all simple distinct and $\Sigma(P) =
\{\beta_1 , \cdots , \beta_k\}$. We prove by induction on $i$ that
$U({-\beta_{i}}) \subset G_Y$.  

Let $y = t \cdot P \in Y$ and let $\varphi : G \to X, g \mapsto g
\cdot y$. We have $U({-\beta_1}) \cdot y = t U({- t^{-1}(\beta_1)})
\cdot P$ and $t^{-1}(\beta_1) = s_{\beta_k} \cdots s_{\beta_{2}}
(-\beta_1)$. Since all the roots $(\beta_i)_{i \in [1,k]}$ are simple
and distinct it follows that $t^{-1}(\beta_1) < 0$. Therefore $U({-
  \beta_1}) \subset \varphi^{-1}(Y)$. Since $Y' = X(s_{\beta_2} \cdots
s_{\beta_{k}}) \subset Y$ it follows that $G_{Y'} \subset G_Y$ and the
result follows by induction. 
\end{proof}

\section{Results on extension of formal functions}

\subsection{Extension of formal functions for Schubert varieties}

In this section we want to extend a connectedness result from \cite{P}
to a result on extension of formal functions. 

\begin{defn}
\label{admis}
Let $w,v\in W^P$, we say that $X^P(v)$ is {\bf admissible} in $X^P(w)$
if $S^P(w)X^P(v)=X^P(w)$ and $\Sigma^P(w)\cap\Sigma^P(v)=\emptyset$. 
\end{defn}

\begin{prop}[{see \cite[Theorem 1.5]{P}}]
\label{Perrin}
Let $f:Y\to X=G/P$ a proper morphism with $Y$ irreducible. Let $w,v\in
W^P$ such that $X^P(v)$ is admissible for $X^P(w)$. Suppose that
$[f(Y)] \cdot [X^P(v)] \neq 0$. Then $f^{-1}(X^P(w))$ is connected.  
\end{prop}

We extend this result to a result on the G3 property.

\begin{thm}\label{main1}
Let $f:Y\to X=G/P$ a proper morphism. Let $w,v\in W^P$ such that
$X^P(v)$ is admissible in $X^P(w)$. Suppose that $[f(Y)] \cdot
[X^P(v)] \ne
0$. Then $f^{-1}(X^P(w))$ is G3 in $Y$. 
\end{thm}

\begin{proof}
Using Propositions \ref{connectfield}, \ref{Badescu} and \ref{Perrin},
we may assume that $Y$ is normal. Let $Q=S^P(v)$ and $Z = \{ ( y ,
\overline{ h } ) \in Y \times G/Q \ | \ f ( y ) \in h ( X^P ( v ) )
\}$. Consider the incidence diagram: 
\[\xymatrix{Z \ar[r]^-p \ar[d]_-q & Y \\
G/Q. & \\}\]
The map $p$ gives $Z$ the structure of a $X^Q(u)$-bundle on $X$, where
$X^Q(u)$ is the closure in $G/Q$ of the $P$-orbit of the Schubert
variety $X^Q(v^{-1})$. Therefore, $p$ is surjective and $Z$ is
irreducible. Furthermore, since $[{f(Y)}] \cdot [X^P(v)] \ne 0$, the
map $q$ is also surjective.

\begin{lemma}\label{lemita}
Let $X^Q(t)$ be a minimal generating Schubert variety in $G/Q$ and let
$W = f^{-1}(X^P(w))$. We have $q^{-1}(X^Q(t))\subset p^{-1}(W)$. 

\end{lemma}

\begin{proof}
The element $t\in W^Q$ has a reduced expression $s_{\beta_1} \cdots
s_{\beta_k}$ such that the roots $(\beta_i)_{i \in [1,k]}$ are all
simple distinct and $\Sigma(Q) = \{\beta_1 , \cdots , \beta_k\}$. The
Schubert variety $X^Q(t)$ is thus the closure of $U(-\beta_1) \cdots
U(-\beta_k) \cdot Q$. 

We prove $f^{-1}(h(X^P(v))\subset f^{-1}(X^P(w)$ for $h\in U(-\beta_1)
\cdots U(-\beta_k)Q$. It is enough to prove that $U(-\beta_1) \cdots
U(-\beta_k) Q \cdot X^P(v) \subset X^P(w)$. Since $X^P(v)$ is
admissible, the sets $\Sigma^P(w)$ and $\Sigma^P(v)=\Sigma(Q)=\{\beta_1,
\cdots ,\beta_k\}$ are disjoint. Therefore, $U(-\beta_i) \subset
S^P(w)$ for al $i \in [1,k]$. Since $Q$ is the stabiliser of $X^P(v)$,
we have: 
\[U(-\beta_1) \cdots U(-\beta_k) Q \cdot X^P(v) = U(-\beta_1) \cdots
U(-\beta_k) \cdot X^P(v) \subset X^P(w).\] 
This concludes the proof.
\end{proof}

Lemma \ref{lemma-gen} implies that $X^Q(t)$ is G3 in $G/Q$ and
Proposition \ref{HiroMatsu} implies that $q^{-1}(X^Q(t))$ is G3 in
$Z$. Consider $g:\overline{Z}\to Z$ the normalization of $Z$ and let
${W} = f^{-1}(X^P(w))$. By Proposition \ref{Perrin}, we have that
$\overline{W} = g^{-1}p^{-1}f^{-1}(X^P(w))$ is connected in
$\overline{Z}$. Therefore,
$K(\overline{Z}_{\overline{W}})=K(Z_{p^{-1}(W)})$ is a field. On the
other hand, $q^{-1}(X^Q(t))\subset p^{-1}(W)$ by Lemma
\ref{lemita}. Applying Proposition \ref{Badescu} we get that
$p^{-1}(W)$ is G3 in $Z$ and that $W$ is G3 in $Y$. 
\end{proof}

\subsection{Application to the diagonal of flag varieties}
\label{section-flag}

In this section, we prove a G3-Bertini type result for the diagonal of
partial flag varieties. 

Given a vector space $V$ and a sequence $(a_i)_{i \in [1,r]}$ of
positive integers, we write $\F((a_i)_{i \in [1,r]} ; V)$ for the
variety parametrising partial flags $ 0 \subset E_1 \subset \cdots
\subset E_r \subset V$ where $E_i$ is a vector subspace of $V$ of
dimension $a_i$. If $\dim V = n$, we will also denote this variety by
$\F((a_i)_{i \in [1,r]} ; n)$. The following result is a variation on
a classical trick (see for example \cite[Th\'eor\`eme 7.1]{debarre}).  

Let $V$ be an $n-$dimensional vector space. Let $X = \F((a_i)_{i \in
  [1,r]} ; V)$. Let $W = V\oplus V$ and $p_1 , p_2 : W \to V$ be the
projections. These projections induce a rational map $p : \F((a_i)_{i
  \in [1,r]} ; W) \to  X \times X$ defined by  
\[(E_1 \subset \cdots \subset E_r) \mapsto (p_1(E_1) \subset \cdots
\subset p_1(E_r) , p_2(E_1) \subset \cdots \subset p_2(E_r)).\] 

\begin{lemma}\label{DTrick}
The fiber of $p$ is isomorphic to 
$$\prod_{i = 1}^r {\rm GL}(a_i - a_{i-1})$$
with $a_0 = 0$.
\end{lemma}

\begin{proof}
Let $(E,F) \in X \times X$ with $E = E_1\subset \cdots \subset
E_r\subset V$ and $F = F_1\subset \cdots \subset F_r\subset V$. A flag
$G = G_1\subset \cdots \subset G_r \subset W$ is in $p^{-1}(E,F)$ if
and only if $p_1(E_i)=G_i$ and $p_2(F_i)=G_i$ for all $i$. In that
case $G_i$ can be seen as the graph in $E_i\times F_i$ of an
isomorphism. Therefore, the whole flag $G$ can be seen as an
isomorphism between $E_r$ and $F_r$ such that the image of $E_i$ is
$F_i$ for all $i<r$. The result follows. 
\end{proof}

Let $X = \F((a_i)_{i \in [1,r]} ; V)$ and let $H$ be a hyperplane of
$V$. Write $\Delta_H$ for the diagonal in $\F((a_i)_{i \in [1,r]} ; H)
\times \F((a_i)_{i \in [1,r]} ; H)$ and $\Delta$ for the diagonal in
$X \times X$. We have an inclusion $\Delta_H \subset \Delta$.  

\begin{thm}
Let $f:Y\to X\times X$ be a proper morphism with $[f(Y)] \cdot
[\Delta_H] \ne 0$. Then $f^{-1}(\Delta)$ is $G3$ in $Y$. 
\end{thm}

\begin{proof}
By Proposition \ref{connectfield} and \ref{Badescu} we may assume that
$Y$ is normal. Let $p : \F((a_i)_{i \in [1,r]} ; W) \to  X \times X$
be the rational map defined above and let $\Delta_W$ be the diagonal
embedding of $V$ in $W = V \oplus V$. Define
$\tilde{\Delta}=\F((a_i)_{i \in [1,r]} ; \Delta_W)$ . We have a
commutative diagram: 
$$\xymatrix{\tilde{Y} \ar[r]^-{\tilde{f}} \ar[d]^-{\tilde{p}}
  \ar@{}[dr]|-{\square}  & \F((a_i)_{i \in [1,r]} ; W) \ar[d]_-p
  \ar@{<-^)}[r] & \tilde{\Delta} \ar[d] \\
{Y} \ar[r]^-{f} & X \times X \ar@{<-^)}[r] & {\Delta} \\}$$
%
%
Since $f$ is proper, so is $\tilde{f}$. Note that $\tilde{\Delta}$ is
a Schubert variety in $\F((a_i)_{i \in [1,r]} ; W)$. Now, let
$\tilde{H}$ be a hyperplane of $\Delta_W$. Then $\F((a_i)_{i \in
  [1,r]} ; \tilde{H})$ is an admissible Schubert variety in
$\tilde{\Delta}$. On the other hand,
$p(\F((a_i)_{i \in [1,r]} ; \tilde{H})) = \Delta_H$ thus $f(Y)$ intersects
$[\Delta_H]$. It follows that $\tilde{f}(\tilde{Y})$ intersects
$\F((a_i)_{i \in [1,r]} ; H)]$. Applying Theorem \ref{main1},
$\tilde{f}^{-1}(\tilde{\Delta})$ is $G3$ in $\tilde{Y}$. 

Since $p$ is surjective and $p(\tilde{\Delta})=\Delta$, we get
$\tilde{p}(\tilde{f}^{-1}(\tilde{\Delta})) = f^{-1}(\Delta)$ by base
extension. Applying Propositions \ref{connectfield} and \ref{Badescu},
we have that $f^{-1}(\Delta)$ is $G3$ in $Y$.
\end{proof}

\begin{remark}
Using the same technique, similar G3 results on the inverse of the
diagonal of isotropic Gra{\ss}mann varieties can be deduced from
connectedness results proved in \cite[Theorem 2.2]{P}.
\end{remark}

\end{document}